\documentclass[11pt]{article}
\usepackage[english]{babel}
\usepackage{amsfonts,upgreek}
\usepackage[utf8]{inputenc}
\usepackage{amsthm}
\usepackage{ mathrsfs }
\usepackage{amsmath}
\usepackage{mathrsfs}
\usepackage{enumitem}
\usepackage{mathtools}
\usepackage{bbm}
\usepackage{makeidx}
\usepackage{graphicx}
\usepackage{frontespizio}
\usepackage{color}
\usepackage[makeroom]{cancel}
\usepackage[normalem]{ulem}
\usepackage{ amssymb }
\usepackage{a4wide}
\usepackage[hidelinks = true]{hyperref}
\usepackage{mathtools}
\usepackage{enumitem}
\usepackage[titletoc]{appendix}
\usepackage{esint,soul}
\mathtoolsset{showonlyrefs=true}

\numberwithin{equation}{section}

\newtheorem{theorem}{Theorem}[section]
\newtheorem*{axiom}{Axiom}

\newtheorem{lemma}[theorem]{Lemma}
\newtheorem{proposition}[theorem]{Proposition}

\newtheorem{remark}[theorem]{Remark}

\newenvironment{proofSTEPS}{\removelastskip\par\medskip   
\noindent{\em Proof via `big steps'} \rm}{\penalty-20\null\hfill$\square$\par\medbreak}

\newenvironment{proofACDC}{\removelastskip\par\medskip   
\noindent{\em Proof via `small steps'} \rm}{\penalty-20\null\hfill$\square$\par\medbreak}

\newcommand{\N}{\mathbb{N}}
\newcommand{\Q}{\mathbb{Q}}
\newcommand{\R}{\mathbb{R}}

\newcommand{\restr}[1]{\lower3pt\hbox{$|_{#1}$}}

\newcommand{\X}{{\rm X}}

\newcommand{\sfd}{{\sf d}}

\renewcommand{\d}{{\rm d}}
		
\newcommand{\nchi}{{\raise.3ex\hbox{$\chi$}}}

\newcommand{\fr}{\penalty-20\null\hfill$\blacksquare$}         

\newcommand{\WO}{{\sf WO}}

\newcommand{\fs}[1]{{\footnotesize #1}}            

\title{Some examples of use of transfinite induction in analysis}

\begin{document}

\author{Nicola Gigli\ \thanks{SISSA, ngigli@sissa.it}  }

\maketitle	
	
\begin{abstract}

It is not uncommon in analysis that   existence of  extremal objects is obtained via an iterative procedure: we start from a given admissible object, then modify it, then  modify again etc... If being extremal means maximimizing a real valued quantity and we are sure to approach the supremum fast enough, after a countable number of steps and a limiting procedure we are done. 

In this short note we want to advertise a slightly different line of thought, where rather than trying to approach the supremum fast enough, we: try to increase, if possible, the function to be maximized and, at the same time, index our recursive procedure over ordinals. Since there are no increasing functions from $\omega_1$ to $\R$, the procedure must stop at some countable ordinal and existence is proved anyway.

The advantage of this line of reasoning is that it can be helpful even in situations where it is not so evident how to measure `being maximal' via a real valued function. This is the case, for instance, for existence of a Maximal Globally Hyperbolic Development of an initial data set in General Relativity.

Speaking of this particular example, we also show that such `real-valued quantification' of the size of a development is actually possible,  thus existence of a maximal one can be obtained in a countable number of steps using the  original argument in  \cite{CBGer69} together with the standard procedure depicted above. This  provides a  way alternative to the one given in \cite{SbDezorn16} to  `dezornify' the proof in  \cite{CBGer69}.

\end{abstract}

\tableofcontents

\section{Introduction}
In many situations in analysis we want to prove existence of a certain extremal object, but perhaps such  existence  is not immediately evident, and we must proceed by approximation. If being extremal means maximizing a real valued function, then a typical procedure would be:  start from a given admissible object, then modify it by trying to approach the supremum, then modify it again etc... If we are sure to approach the supremum fast enough, after a countable number of steps and a limiting procedure we are done. We shall informally call this `proof via big steps'.

In this paper we want to advertise a different kind of argument, that somehow replaces these `big steps' with possibly `small ones'. It relies on the fact that any monotone map from the first uncountable ordinal to the reals must be eventually constant (we recall the proof of this basic fact in Proposition \ref{prop:omega1r}), thus if we index our iterative procedure over the ordinals and at every step we try to augment our given real valued function, then regardless of how big our steps are, we are sure that after a countable number of steps we are done anyway.

We shall illustrate this principle with three practical applications, presented in a somehow increasing order of difficulty.

There is a clear downside in our approach, namely the fact that it relies on concepts such as ordinal numbers and transfinite recursion, that are rarely found in analysis' papers. The upside, though, is that it might work even in situations where it is not really clear `how to quantify progresses' in our construction. This is the case, in particular , for existence of a the Maximal Global Hyperbolic Development of an initial data set, and in this sense it is no surprise that the original proof in \cite{CBGer69} relied on the Axiom of Choice. In this setting, the key observation that allows to reduce the number of Choices to a countable quantity is the statement by Geroch \cite{Ger68}:
\begin{equation}
\label{eq:ger}
\text{A smooth connected manifold carrying a non-degenerate metric tensor is separable}
\end{equation}
whose proof we shall report in Proposition \ref{prop:geroch}. We think at this statement as: a non-degenerate metric tensor gives, regardless its signature, a `scale measured by reals' to the given manifold, and once such a scale is present, the object must be separable.

With all this being said, we shall also provide an alternative proof of the existence of a Maximal Globally Hyperbolic Development that has nothing to do with transfinite recursion and ordinals. Rather, we construct a way of quantifying the size of a development, see \eqref{eq:quantMGHD}, so that a maximal one can be built via `big steps', along classical lines. Still, the fact that this quantification is not-so-evidently found is perhaps an argument in favour of the relevance of the line of reasoning discussed here in the mathematical conversation.

\bigskip

In conclusion, the main takeaways of this paper are:
\begin{enumerate}
\item The illustration of the potential use of transfinite induction/recursion over countable ordinals to deduce `honest' results in geometric analysis. The advantage, if any, over more standard arguments is that it permits to pay less attention to the choices made during the induction procedure. We stress that this line of thought is based over ${\sf DC}_{\omega_1}$ (i.e.\ Dependent Choice indexed over the first uncountable ordinal $\omega_1$) that in turn is still a Choice strong enough to imply, for instance, existence of non-measurable subsets of the reals.
\item A  modification of the original proof of existence of Maximal Global Hyperbolic Developments, based on the quantification \eqref{eq:quantMGHD}. This has nothing to do with ordinals and transfinite induction and because of this might be of some interest for the mathematical GR community.

The use of the quantification \eqref{eq:quantMGHD} eliminates the need of Zorn's lemma replacing it with the more commonly used Countable Dependent Choice. In this sense this part of the work provides an alternative way to `dezornify' the original proof in \cite{CBGer69}, comparable to that given in \cite{SbDezorn16}. We remark that in \cite{WWWY13} a further different proof has been given, that relies on no Choice at all for what concerns the construction part of MGHD. Our proof, much like several other proofs by iteration, does not allow to fully drop Choice and we shall often use the usual Countable Dependent Choice throughout the text. We do not regard this as a weakness of our approach, as we believe anyway -- in line with \cite{BernaysCDC} --  that CDC is crucially needed in most of the aspects of (geometric) analysis. 
\end{enumerate}

\bigskip

I thought for the first time at a `small steps' argument via transfinite recursion of the kind just discussed in my PhD thesis \cite[Proposition 5.10]{GigliThesis}. I mentioned this to a number of colleagues and to all of them the argument seemed new; still, I cannot exclude that in circles far from mine this is a well-known line of reasoning. In \cite{GigliThesis} I needed a statement closely resembling Ekeland's variational principle, of which I was unaware, and admittedly I did not find a way to make a proper quantification, thus I resorted on this other approach. Later, any time I was tempted to use this  line of thought, a quantification was easily found, making the small-steps argument irrelevant. I thus never published anything on the matter, limiting myself to asking questions on the matter on MathOverflow \cite{ACDCMO}, especially in connection to the kind of Axiom of Choice needed to carry out the argument.

I realized only recently the relevance of this line of thought in relation to Maximal Global Hyperbolic Developments, and stimulated by the above considerations looked for  a way to measure the size of developments via a real valued function.

\bigskip

{\bf Acknowledgment} I wish to warmly thank Prof. A. Karagila for a series of conversations that helped me understanding the equivalence of  ${\rm DC}_{\omega_1}$  and Lemma \ref{le:ACDC}.

This study was funded by the European Union - NextGenerationEU, in the framework of the PRIN Project Contemporary perspectives on geometry and gravity (code 2022JJ8KER – CUP G53D23001810006). The views and opinions expressed are solely those of the author and do not necessarily reflect those of the European Union, nor can the European Union be held responsible for them.

I also want to thank the referees for the careful reading of the manuscript and their useful suggestions,  Prof.\ A. Lerario for having brought to my attention the Pr\"ufer surface, relevant in Remark \ref{re:longest}, and also Prof.\ Y.\ Shlapentokh-Rothman for stimulating conversations on the matter.

\section{Abstract presentation of the mechanism}\label{se:axiom}
Here we present the abstract mechanism behind our use of transfinite induction. We assume the reader familiar with the basic theory of ordinals, as can be found for instance in \cite{Jech2003}. An informal, but rigorous, presentation of all the necessary ingredients is also given in the appendix.

We shall use greek letters $\alpha,\beta,\gamma\ldots$ to denote ordinals. $\omega_1$ is the first uncountable ordinal. In the discussion below a `sequence' in a set $A$ is a map from some ordinal to $A$. It will be either denoted by  $(a_{\alpha'})_{\alpha'<\alpha}$ (or $(a_{\alpha'})_{\alpha'\leq \alpha}$ if it is also defined at $\alpha$) or by $f:\alpha\to A$ (according to the fact that an ordinal is the set of the smaller ordinals). In this latter case, for $\beta<\alpha$ the restriction of $f$ to $\beta$ (i.e.\ the first $\beta$ terms in the sequence) is denoted $f\restr\beta$. If the domain of a sequence is the ordinal $\alpha$, we say that it has length $\alpha$, or that it is an $\alpha$-sequence.

\bigskip

The technique we want to discuss is for proving existence results. So we are given a non-empty set $A$ and a subset $F\subset A$ of `final' elements that we intend to show to be non-empty. The way we have to probing that $F$ is not empty is via a set  $\mathcal S$ of sequences in $A$ that, roughly said, have the property that `if a sequence in $\mathcal S$ does not meet $F$, then it can be extended'. If this occurs, then indeed at least one sequence in $\mathcal S$ must meet $F$, which therefore is not empty. Even more: if we have an upper bound on the length of sequences in $\mathcal S$ (in our case this is $\omega_1$), then $F$ must be met before reaching such upper bound. 

We emphasize that for this argument to work it is imperative to work with sequences indexed by ordinals, as easy examples show that ordinary sequences cannot lead to the same conclusion.
\begin{lemma}\label{le:ACDC} Let  $A$  be a set, $F\subset A$ a subset and $\mathcal S$ a non-empty collection of sequences in $A$ of length $<\omega_1$ with the following properties:
\begin{itemize}
\item[i)] {\rm (if $F$ has not been reached, then we can continue)}\\  If $(a_{\alpha'})_{\alpha'<\alpha}\in\mathcal S$ is so that $a_{\alpha'}\notin F$ for every $\alpha'<\alpha$, then there is $a_\alpha\in A$ such that  $(a_{\alpha'})_{\alpha'\leq\alpha}\in\mathcal S$;
 
\item[ii)] {\rm (limits of admissible sequences are admissible)} \\
 If $\alpha<\omega_1$ is a limit ordinal and $(a_{\alpha'})_{\alpha'<\alpha}$ is so that $(a_{\alpha''})_{\alpha''<\alpha'}\in\mathcal S$ for every $\alpha'<\alpha$, then  $(a_{\alpha'})_{\alpha'<\alpha}\in\mathcal S$;
 
\item[iii)] {\rm (there is no admissible sequence $\omega_1$-long)}\\
 There is no $(a_\alpha)_{\alpha<\omega_1}$ such that for every $\alpha<\omega_1$ we have $(a_{\alpha'})_{\alpha'<\alpha}\in\mathcal S$.
\end{itemize}
Then there is $(a_{\alpha'})_{\alpha'<\alpha}\in\mathcal S$ and $\bar\alpha<\alpha$   so that    $a_{\bar\alpha}\in F$. In particular, $F$ is not empty.
\end{lemma}
\begin{proof}
Define recursively a transfinite sequence as follows. Let $(a_{\alpha'})_{\alpha'<\alpha_0}\in \mathcal S$ be arbitrary. If this satisfies the conclusion we are done. Otherwise we look for extensions in $\mathcal S$ of this sequence.  Then let  $\alpha<\omega_1$ be with $\alpha_0<\alpha$ and assume to have already  defined $a_{\alpha'}$ for every $\alpha_0\leq{\alpha'<\alpha}$ and that  $(a_{\alpha''})_{\alpha''\leq \alpha'}\in\mathcal S$ for every $\alpha'<\alpha$. Then $(a_{\alpha'})_{\alpha'<\alpha}\in\mathcal S$: this is obvious if $\alpha$ is a successor, while if it is limit we use   $(ii)$. If for some $\alpha'<\alpha$ we have $a_{\alpha'}\in F$ we are done, otherwise  by $(i)$ there is $a_\alpha\in A$ such that  $(a_{\alpha'})_{\alpha'\leq\alpha}\in\mathcal S$.

Proceed by transfinite recursion until either we find an $\omega_1$-long sequence  $(a_\alpha)_{\alpha<\omega_1}$ such that for every $\alpha<\omega_1$ we have $(a_{\alpha'})_{\alpha'<\alpha}\in\mathcal S$ or we find some $\alpha<\omega_1$ and  $\bar\alpha<\alpha$  such that $(a_{\alpha'})_{\alpha'<\alpha}\in\mathcal S$    so that    $a_{\bar\alpha}\in F$. Assumption $(iii)$ tells that the former case does not happen, thus the latter holds.
\end{proof}

We might  replace assumption $(iii)$ above with the assumption that sequences in $\mathcal S$ are injective, as this is what happens in the applications we have in mind. However, for us it might be relevant to know that the recursive procedure only involves countable ordinals (e.g.\ in discussing the application to the Hahn-Jordan decomposition we need to be sure that the sets we produce at each step are measurable) and for this assumption $(iii)$ matters. A further advantage of having it is that it clarifies which version of Axiom of Choice is needed to carry on the argument: as Prof. A. Karagila patiently explained me, Lemma \ref{le:ACDC} is equivalent to  ${\rm DC}_{\omega_1}$, i.e.\ Dependent Choice indexed over $\omega_1$ (we refer to the comprehensive monograph \cite{JechChoice}, and in particular Chapter 8 in there, for more on the matter. We illustrate this by first recalling what ${\rm DC}_{\omega_1}$ says. Recall that given a set $A$, by $A^{<\omega_1}$ it is meant the collection of $\alpha$-sequences in $A$ for $\alpha<\omega_1$.
\begin{axiom}[${\sf DC}_{\omega_1}$] Let $A$  be a non-empty set  and $G$ a map from $A^{<\omega_1}$ to non-empty subsets of $A$. Then there exists an $\omega_1$-sequence $f:\omega_1\to A$ such that 
\begin{equation}
\label{eq:dc1}
f(\alpha)\in G(f\restr\alpha)
\end{equation}
holds for every  $\alpha<\omega_1$.
\end{axiom}
Intuitively, if one wants to build $f$ as in the statement, then she must Choose $f(0)$ among elements of $G(\emptyset)$, then $f(1)$ among the elements of $G(f\restr{\{0\}})$, \dots, then $f(n)$ among the elements of $G(f\restr{\{0,\ldots,n-1\}})$,\dots then $f(\omega)$ among the elements of $G(f\restr{\{0,1,\ldots\}})$ etc... At any given step, the existence of an admissible choice is trivial, as  $G$ takes value in non-empty subsets of $A$. The Axiom ${\sf DC}_{\omega_1}$ ensures that all these subsequent Choices up to $\omega_1$ are actually possible. We have:
\begin{proposition} In Zermelo-Fraekel set theory, Lemma \ref{le:ACDC} is equivalent to Axiom ${\sf DC}_{\omega_1}$. 
\end{proposition}
\begin{proof}\ \\
\emph{Lemma \ref{le:ACDC} implies Axiom ${\sf DC}_{\omega_1}$}. We argue by contradiction. We thus assume that ${\sf DC}_{\omega_1}$ does not hold, i.e.\ that for $A$ and $G$ as in the statement of the axiom there is no $f:\omega_1\to A$ such that  \eqref{eq:dc1} holds for every $\alpha<\omega_1$.

Let $F$ be the empty set and $\mathcal S$ be the collection of sequences $f$ valued in $A$ such that \eqref{eq:dc1} holds for every $\alpha$ smaller than the length of the sequence.  We verify that $A,F,\mathcal S$ verify $(i),(ii),(iii)$ in Lemma \ref{le:ACDC}. Property $(i)$ is obvious from the fact that $G$ takes values in non-empty subsets of $A$, $(ii)$ is obvious by definition and $(iii)$ is our assumption of failure of ${\sf DC}_{\omega_1}$. We thus see that Lemma \ref{le:ACDC} fails, because the data satisfy the assumption but the set $F$ is empty, contradicting the conclusion. 

\emph{Axiom ${\sf DC}_{\omega_1}$ implies Lemma \ref{le:ACDC}}.   Let $A,F,\mathcal S$ be as in the statement of Lemma \ref{le:ACDC}. Let $\mathcal S'$ be the collection of sequences in $\mathcal S$ and of their restriction to ordinals smaller than their original lengths.  Let us pick  one extra element, call it $\bar a$, not belonging to $A$ and let us define $\bar A:=A\cup\{\bar a\}$ and $\bar F:=F\cup\{\bar a\}$. We shall think $\mathcal S,\mathcal S'$ as set of maps from countable ordinals into $\bar A$. Now for $f\in \bar A^{<\omega_1}$ we define $G(f)$ as follows:  if $f$ is in  $\mathcal S'$ and the image does not meet $F$, then $G(f)\subset A\subset\bar A$ is the collection of those values so that if we extend $f$ one further step via this value, the extension is still in $\mathcal S'$ (notice that by definition of $\mathcal S'$ and assumption $(i)$ in Lemma \ref{le:ACDC}, in this case $G(f)$ is not empty).  In all the other cases we put $G(f):=\{\bar a\}$.

The Axiom ${\sf DC}_{\omega_1}$ ensures that there is   $f:\omega_1\to\bar A$ such that \eqref{eq:dc1} holds for every $\alpha<\omega_1$. Suppose that $f(\alpha)\in A\setminus F$ for every $\alpha<\omega_1$. Then we would have $f\restr\alpha\in\mathcal S'$ for every $\alpha<\omega_1$ (as shown via an easy argument by transfinite induction also based on Assumption $(ii)$ - notice that this does not need any Choice at all), contradicting  Assumption $(iii)$ in Lemma \ref{le:ACDC} (and the definition of $\mathcal S'$). Therefore there is a least $\alpha<\omega_1$ so that $f(\alpha)\in F\cup\{\bar a\}$. If there is $\beta<\alpha$ such that $f\restr\beta\notin\mathcal S'$, then $f(\beta)=\bar a$, contradicting the definition of $\alpha$. Hence $f\restr\beta\in\mathcal S'$ for every $\beta<\alpha$. We now distinguish two cases. 

If $\alpha$ is a limit ordinal, then Assumption $(ii)$ tells that $f\restr\alpha\in\mathcal S'$ and since the image of $f\restr\alpha$ does not meet $F$, we see that $G(f\restr\alpha)\subset A$. Hence $f(\alpha)\in A$ and thus $f(\alpha)\in F$. In this case the conclusion of Lemma \ref{le:ACDC} holds.

If instead $\alpha$ is a successor ordinal, say $\alpha=\tilde\alpha+1$, then $f\restr{\tilde\alpha}\in\mathcal S'$ and  the image of $f\restr{\tilde\alpha}$ does not meet $F$, so  $f(\tilde\alpha)\in G(f\restr{\tilde\alpha})$ is chosen, by definition of $G$, in such a way  that $f\restr\alpha\in\mathcal S'$. This and the fact that the image of $f\restr\alpha$ does not meet $F$ (by definition of $\alpha$ and the fact that  the image of  a sequence in $\mathcal S'$ is contained in $A$) implies that $G(f\restr\alpha)\subset A$, proving also in this case that $f(\alpha)\in F$ and thus that the conclusion of Lemma \ref{le:ACDC} holds.
\end{proof}

%
%
%

\section{Three examples}

\subsection{The Hahn-Jordan decomposition}
The Hahn-Jordan decomposition is a basic statement in measure theory telling that a finite signed measure $\mu$ can uniquely be written as difference of two non-negative signed measures concentrated on disjoint subsets. Its proof relies on the following lemma. Recall that given a finite signed measure $\mu$, a measurable set is called `negative' if it has no measurable subset with positive measure.

\begin{lemma}
Let $\mu$ be a finite signed measure on the measurable space $(\X,\mathcal A)$ and $E\in \mathcal A$.

Then there is $E'\subset E$, $E\in\mathcal A$, negative with $\mu(E')\leq\mu(E)$.
\end{lemma}

\begin{proofSTEPS}
For every integer $n\in\N$ we define a set $E_n\in\mathcal A$ as follows. Let $E_0:=E$ and recursively define $t_n:=\sup_{C\subset E_n} \mu(C)$, then pick  $C_n\subset E_n$ so that $\mu(C_n)\geq\min\{1,\tfrac{t_n}2\}$  and let $E_{n+1}:=E_n\setminus C_n$.

Put $E_\infty:=\cap_nE_n=E\setminus \cup_nC_n$ and notice that  $\mu(E_\infty)=\mu(E)-\sum_n\mu(C_n)\leq\mu(E)$. Suppose $E_\infty$ is not negative. Then there is $C\subset E_\infty\subset E_n$ for every $n$ with $\mu(C)>0$. Thus  $t_n>\mu(C)$ for every $n$ and therefore $\mu(C_n)\geq\min\{1, \tfrac{\mu(C)}2\}$. Since the $C_n$'s are disjoint we would get  $\mu(\cup_nC_n)=\sum_n\mu(C_n)=+\infty$, contradicting the fact that $\mu$ is finite.
 \end{proofSTEPS}
 \begin{proofACDC} We shall use Lemma \ref{le:ACDC} with $A:=\mathcal A$, $F\subset A$ the collection of negative subsets $E'$ of $E$ with $\mu(E')\leq\mu(E)$ (equivalently: of those subsets $E'$ of $E$ so that for no $E''\subset E'$ we have $\mu(E'')<\mu(E')$)   and $\mathcal S$ the collection of decreasing sequences $(E_\beta)_{\beta<\alpha}$ of subsets of $E$ with $E_0:=E$ and   $\beta\mapsto \mu(E_\beta)$ strictly decreasing. The fact that these satisfy properties $(i),(ii)$ in Lemma \ref{le:ACDC} obvious. Property $(iii)$ follows from the  fact that there are no strictly increasing maps from $\omega_1$ to $\R$. 
 
Lemma \ref{le:ACDC}  grants that $F$ is not empty, which is the claim.
 \end{proofACDC}

\subsection{Ekeland's variational principle}
The variational principle of Ekeland allows to show existence of almost minimizers of a lower semicontinuous functional in absence of compactness (but in presence of completeness).

A possible formulation is the following:
\begin{theorem}\label{cor:cor}
Let $(\X,\sfd)$ be a complete metric space, $f:\X\to[0,+\infty]$ be lower semicontinuous and $\bar x\in\X$. Then there is $\bar y\in\X$ with $f(\bar y)+\sfd(\bar x,\bar y)\leq f(\bar x)$ such that
\begin{equation}
\label{eq:quasiminimo}
f(z)+\sfd(z,\bar y)\geq f(\bar y)\qquad\forall z\in\X.
\end{equation}
\end{theorem}
\begin{proofSTEPS}
Define the relation $\leq$ on $\X$ by declaring that $z_1\leq z_2$ whenever 
\[
f(z_1)+\sfd(z_1,z_2)\leq f(z_2).
\]
It is clear that this is a partial order and by lower semicontinuity of $f$ that $\{y\in\X:y\leq x\}$ is closed for every $x\in\X$. Now consider the given $\bar x$ and notice that if there is no $x\leq\bar x$ with $f(x)<\infty$, then the choice $\bar y:=\bar x$ satisfies the conclusion, otherwise replacing $\bar x$ with such $x$ we can assume that $f(\bar x)<\infty$.

For every $x\in\X$ define
\[ 
i(x):=\inf_{y\leq x}f(y),\qquad\text{ and }\qquad D(x):=f(x)-i(x)\geq 0
\]
and recursively define a sequence $(x_n)\subset\X$ by putting $x_0:=\bar x$ and then, given $x_n$, find $x_{n+1}\leq x_n$ with $f(x_{n+1})\leq \tfrac12(i(x_n)+f(x_n))$. It is clear that such $x_{n+1}$ can be found, that $\sfd(x_{n+1},x_n)\leq D(x_n)$ and that $D(x_{n+1})\leq\tfrac12 D(x_n)$. Since $D(x_0)<\infty$ we conclude that $(x_n)$ is Cauchy, hence it converges to a limit $x_\infty$. The lower semicontinuity of $f$ and the fact that $i$ is $\leq$-non-decreasing show that $D(x_\infty)=0$. This means that $f(x_\infty)=i(x_\infty)$ and thus that $\bar y:=x_\infty\leq\bar x$ is the desired point.
\end{proofSTEPS}
\begin{proofACDC}
Define the relation $\leq$ on $\X$ by declaring that $z_1\leq z_2$ whenever 
\[
f(z_1)+\sfd(z_1,z_2)\leq f(z_2).
\]
It is clear that this is a partial order and by lower semicontinuity of $f$ that $\{y\in\X:y\leq x\}$ is closed for every $x\in\X$. Now consider the given $\bar x$ and notice that if there is no $x\leq\bar x$ with $f(x)<\infty$, then the choice $\bar y:=\bar x$ satisfies the conclusion, otherwise replacing $\bar x$ with such $x$ we can assume that $f(\bar x)<\infty$.

We use Lemma \ref{le:ACDC} with $A:=\X$, $F\subset A$ the collection of points $\leq\bar x$ that are $\leq$-minimizers, i.e.\ points $x$ so that $y\leq x$ implies $y=x$, and $\mathcal S$ the collection of continuous, injective and $\leq$-non-increasing sequences bounded from above by $\bar x$. It is clear that assumption $(ii)$ of Lemma \ref{le:ACDC}  are satisfied. $(iii)$ is satisfied as well because along any sequence in $\mathcal S$ the function $f$ is real valued, monotone and injective, and we already recalled that no such map can be defined on $\omega_1$. It thus remains to show $(i)$. This is obvious if $\alpha$ is a successor ordinal, as in this case any $\alpha$-sequence $(x_\beta)_{\beta<\alpha}$ has $x_{\alpha-1}$ as last element, and if it is not $\leq$-minimizer, an $x_\alpha\lneq x_{\alpha-1}$ exists. Thus say that $\alpha$ is a limit ordinal. We claim that for every  $\alpha$-sequence $(x_\beta)_{\beta<\alpha}\in\mathcal S$ we have
\begin{equation}
\label{eq:cauchyord}
f(x_\beta)+\sfd(x_\beta,x_\gamma)\leq f(x_\gamma)\qquad\forall \gamma\leq\beta<\alpha.
\end{equation}
 We show this by transfinite induction on $\beta$ ($\alpha$ is fixed here and can be assumed to be $>0$). It is clear that \eqref{eq:cauchyord} holds for $\beta=0$ and that its validity for $\beta$ implies that for $\beta+1$. For $\beta$ limit, let $(\beta_n)$ be an increasing sequence whose supremum is $\beta$
and notice that \eqref{eq:cauchyord} and the inductive assumption show that $n\mapsto x_{\beta_n}$ is a Cauchy sequence: its limit must, by continuity of $(x_\cdot)$, coincide with $x_\beta$. Also, by lower semicontinuity of $f$ we see that \eqref{eq:cauchyord} holds for $\beta$.

Thus \eqref{eq:cauchyord} holds. Now let   $(\alpha_n)$ be increasing with supremum $\alpha$. As before, \eqref{eq:cauchyord} implies that $n\mapsto x_{\alpha_n}$ is Cauchy, hence admits a limit and this limit does not depend on the sequence $(\alpha_n)$ chosen, as interlacing any two such sequences we still have that the limit exists. Call $x_\alpha$ such limit and notice that, by construction and arguments already used, the sequence $(x_\beta)_{\beta\leq \alpha}$ belongs to $\mathcal S$, as desired.

It follows by Lemma \ref{le:ACDC}    that $F$ is not empty, which was the claim.
\end{proofACDC}

\subsection{Maximal Globally Hyperbolic Development}

An \emph{initial data set} is a triple $(\Sigma,h,\kappa)$ with $\Sigma$ being a smooth 3-dimensional manifold, $h$ a smooth Riemannian metric on it and $\kappa$  a symmetric 2-form satisfying the constraint equations
\[
\begin{split}
R+({\rm tr}_h\kappa)^2-\|\kappa\|^2_h&=0,\\
{\rm div}_h(\kappa)-\d({\rm tr}_h\kappa)&=0.
\end{split}
\]
A \emph{development} of $(\Sigma,h,\kappa)$ is a triple $(M,g,\varphi)$ where:
\begin{itemize}
\item[i)] $(M,g)$ is a Ricci-flat time-oriented spacetime,
\item[ii)] $\varphi:\Sigma\to M$ is a smooth isometric diffeomorphism of $\Sigma$ and $S:=\varphi(\Sigma)\subset M$,
\item[iii)] $\varphi$ sends $\kappa$ to the second fundamental form of $S$,
\item[iv)] $S$ is a Cauchy hypersurface of $M$.
\end{itemize}
In particular, $M$ solves Einstein's equations in the vacuum. Also, it admits a Cauchy hypersurface and thus is globally hyperbolic. For the terminology we refer to \cite{CBGer69}. We observe that the embedding $\varphi:\Sigma\to M$ induces, for every $p\in\Sigma$, a unique causal isometry $\Phi_p$ from $(\R\times T_p\Sigma,\d t^2-h)$ to $(T_{\varphi(p)}M,g)$ sending $(0,v)$ to $\d\varphi(v)$, being intended that $(1,0)\in\R\times T_p\Sigma$ is future directed and is therefore sent to the future-directed normal to $S$ at $\varphi(p)$.

Following \cite{CBGer69}, we say that a development $(\tilde M,\tilde g,\tilde \varphi)$ is an extension of another development $( M, g,\varphi)$ provided there is an injective smooth causal isometry $\psi: M\to\tilde M$ such that $\psi\circ\varphi=\tilde\varphi$. Notably, such $\psi$ is unique and   satisfies
\begin{equation}
\label{eq:phip}
\d\psi_{\varphi(p)}\circ\Phi_p=\tilde\Phi_p\qquad\forall p\in\Sigma.
\end{equation}
We discuss uniqueness.  The fact that any smooth causal isometry $\psi:M\to\tilde M$ must satisfy \eqref{eq:phip} is obvious, and so is the fact that \eqref{eq:phip} fully characterizes $\d\psi:T_{\varphi_p}M\to T_{\tilde\varphi(p)}\tilde M$. Now for uniqueness of $\psi$ let $q\in M$ and $\gamma$ an inextendible timelike geodesic (intended as solution of the geodesic equation) passing through $q$. Since $\varphi(\Sigma)$ is a Cauchy hypersurface of $M$, $\gamma$ must cross it at some time that, up to translation, we can assume to be 0. Let then $T\in\R$ be so that $\gamma_T=q$ and notice that the curve $t\mapsto\tilde\gamma_t:=\psi(\gamma_t)\in\tilde M$ is a geodesic in $\tilde M$, thus its value $\psi(q)$ at $t=T$ is fully determined by $\tilde\gamma_0$ and  $\tilde\gamma_0'$. Since we have $\tilde\gamma_0=\tilde\varphi(\varphi^{-1}(\gamma_0))$ and, by \eqref{eq:phip}, $\tilde\gamma'_0=\tilde\Phi_p(\Phi_p^{-1}(\gamma_0'))$, where $p:=\varphi^{-1}(\gamma_0)$, we see that $\tilde\gamma_0$ and  $\tilde\gamma_0'$ only depend on $\gamma$, $\varphi$ and $\tilde\varphi$. In particular, they do not depend on $\psi$, proving the desired uniqueness.

We shall write $(M_1,g_1,\varphi_1)\preceq (M_2,g_2,\varphi_2)$, or simply $M_1\preceq M_2$, if $(M_2,g_2,\varphi_2)$ extends $(M_1,g_1,\varphi_1)$. The uniqueness just proved shows that this is a `partial order up to unique isomorphism', i.e.\ that $M_1\preceq M_2$ and $M_2\preceq M_1$ imply that the  maps $\psi:M_1\to M_2$ and $\tilde\psi:M_2\to M_1$ associated to these extensions are global causal isometric diffeomorphisms, one the inverse of the other (as by uniqueness $\tilde\psi\circ\psi$ must be the identity on $M_1$ and $\psi\circ\tilde\psi$ that on $M_2$). When this happens we write $M_1\cong M_2$.

We also notice that:
\begin{equation}
\label{eq:podev}
M_1\preceq M_2\text{ and }M_1\not\cong M_2\qquad\Rightarrow\qquad M_2\setminus \overline{\psi(M_1)}\text{ is not empty.}
\end{equation}
Indeed, by assumption there is $p\in M_2\setminus\psi(M_1)$ and by symmetry we can assume that $p$ is in the future of $\varphi_2(\Sigma)$. We claim that the chronological future $I^+(p)$ of $p$ in $M_2$, that is open and not empty, does not meet $\psi(M_1)$. Say otherwise, let $q\in I^+(p)\cap\psi(M_1)$, let $\gamma:[0,1]\to M_2$ be a timelike curve from $p$ to $q$ and $(a,b)\subset[0,1]$ be the maximal open subinterval such that $q\in\gamma((a,b))\subset \psi(M_1)$. Then $\gamma(a)\notin\psi(M_1)$ and thus the curve $\psi^{-1}\circ\gamma:(a,b)\to M_1$ lives in the future of $\varphi_1(\Sigma)$, is past-inextendible and does not meet $\varphi_1(\Sigma)$, contradicting the assumption that $\varphi_1(\Sigma)$ is a Cauchy hyersurface of $M_1$.

The local existence and uniqueness of developments  can be stated as:
\begin{theorem}\label{thm:local}
Let $(\Sigma,h,\kappa)$ be an initial data set. Then it admits a development. Also, if $(M_1,g_1,\varphi_1)$ and  $(M_2,g_2,\varphi_2)$ are two such developments, then they are both extension of a common development, i.e. there is a development  $(M,g,\varphi)$ with $M\preceq M_1$ and $M\preceq M_2$.
\end{theorem}
In \cite{CBGer69} the authors started from this result and, via an argument based on Zorn's lemma, deduced existence and uniqueness of a maximal global hyperbolic development (see Theorem \ref{thm:main} below for the precise meaning of this). We are going to show that in fact for their very same proof to work Lemma  \ref{le:ACDC} (and thus ${\rm DC}_{\omega_1}$) suffices.

The upper bound on the (ordinal) number of choices needed comes from the following proposition, that is a restatement of a result obtained by Geroch in \cite{Ger68}. We report (a minor modification of) the proof to emphasize that the signature of the non-degenerate metric tensor plays no role.

Before coming to the statement, let us agree that a \emph{smooth manifold} is a Hausdorff topological space that is locally homeomorphic to $\R^d$ for some fixed $d\in\N$ and so that the `change of coordinates' are $C^\infty$. Notice that we are not insisting on manifolds to be separable, and in fact
\[
\text{there are examples of connected smooth manifolds that are non-separable,}
\]
the typical one being the `long line', see Remark \ref{re:longest} for more on this.

With this said, we have:
\begin{proposition}\label{prop:geroch}
Let $M$ be a connected smooth manifold and $g$ a non-degenerate and smooth metric tensor on it.

Then $M$ is separable.
\end{proposition}
\begin{proof}
The key of the proof is in existence of the exponential map and in its local invertibility, properties that have nothing to do with the potential lack of separability. 

Since $g$ is smooth and non-degenerate, in any smooth coordinate system the Christoffel symbols $\Gamma^k_{ij}=\tfrac12g^{km}(\tfrac{\partial g_{mj}}{\partial x^i}+\tfrac{\partial g_{im}}{\partial x^j}-\tfrac{\partial g_{ij}}{\partial x^m})$ are well defined and smooth. Hence the geodesic equation $\tfrac{\d^2\gamma^k}{\d t^2}+\Gamma^k_{ij}\frac{\d\gamma^i}{\d t}\frac{\d\gamma^j}{\d t}=0$ admits, for any initial datum, unique smooth solutions existing for some positive time. For any $x\in M$ we denote by $D_x\subset T_xM$ the domain of the exponential map, i.e.\ the collection of those $v\in T_xM$ such that there is $\gamma:[0,1]\to M$ with $\gamma_0=x$, $\gamma_0'=v$ locally solving the geodesic equation in coordinates. Being such $\gamma$  unique, the definition $\exp_x(v):=\gamma_1$ is well posed and an argument based on smooth dependence of solutions w.r.t.\ initial data shows that $DM:=\{(x,v)\in TM:v\in D_x\subset T_xM\}\subset TM$ is open and that $\exp:DM\to M$ is smooth. For the same reason, the set $U_x\subset D_x$ of those $v$'s such that the differential of $\exp_x$ at $v$ is not singular is also open. Then the inverse function theorem grants that $\exp_x:U_x\to M$ is a local diffeomorphism and in particular that $\mathcal U_x:=\exp_x(U_x)\subset M$ is open. Moreover, since $(\d\exp_x)(0)$  is, by the very definition of $\exp_x$, the identity, the inverse function theorem again grants that
\begin{equation}
\label{eq:dainv}
\forall x\in M\text{ there is a neighbourhood $\mathcal V_x\subset M$ of $x$ such that  $x\in\mathcal U_y$ for any $y\in\mathcal V_x$.}
\end{equation}
Now fix $\bar x\in M$ and recursively define $M_n\subset M$ as: $M_0:=\mathcal U_{\bar x}$ and given $M_n$ put $M_{n+1}:=\cup_{x\in M_n}\mathcal U_x$. Put $M_\infty:=\cup_nM_n$. Since the $\mathcal U_x$'s are open, so is $M_\infty$. Also, since $\mathcal U_{\bar x}\subset T_{\bar x}M$ is separable and $\exp_x:U_x\to\mathcal U_x$ is continuous, we see that $M_0=\mathcal U_{\bar x}$ is separable. Then an induction argument based on the separability of $TM_n$ and the continuity of $\exp$ shows that $M_n$ is separable for every $n\in\N$. Hence $M_\infty$ is also separable.

Since $M$ is connected and $M_\infty\subset M$ open and separable, to conclude it suffices to prove that $M_\infty$ is also closed. Thus let $x\in\overline{M}_\infty$  and use \eqref{eq:dainv} to find $\mathcal V_x\subset M$ with the stated properties. In particular, $\mathcal V_x\cap M_{\infty}\neq\emptyset$ and we can find $y\in V_x\cap M_{\infty}$ and thus $n\in\N$ so that $y\in \mathcal V_x\cap M_{n}$. Property \eqref{eq:dainv} then grants that $x\in\mathcal U_y$, hence  $x\in M_{n+1}\subset M_\infty$, as desired.
\end{proof}
\begin{remark}{\rm
In the above statement we used regularity of $g$  to be able to define the exponential map. 

We suspect the same conclusions hold even for $g$ that is merely continuous.
}\fr\end{remark}
We turn to the main proof of the section:
\begin{theorem}\label{thm:main}
Let $(\Sigma,h,\kappa)$ be an initial data set. Then there exists a unique maximal development $(M,g,\varphi)$ of it. This means that:
\begin{itemize}
\item[1)] (Maximality) If $(M',g',\varphi')$ is another development, then $M'\preceq M$;
\item[2)] (Uniqueness) If $(\tilde M,\tilde g,\tilde \varphi)$  is another development satisfying (1) above, then $\tilde M\cong  M$.
\end{itemize}
\end{theorem}
\noindent\emph{Preliminary considerations}
The uniqueness claim is a direct consequence of the definition of maximality  and the fact that $\preceq$ is a partial order on (isomrphism classes of) developments: if $M,\tilde M$ are developments satisfying $(1)$, then $\tilde M\preceq M$ and $M\preceq \tilde M$, proving that $\tilde M\cong M$.

For existence, let us briefly recall how the original proof in \cite{CBGer69} works:
\begin{itemize}
\item[{\sf A)}] One proves that there exists a development  $(M,g,\varphi)$  that is maximal in the sense, weaker than $(1)$, that:
\begin{equation}
\label{eq:altmax}
\text{if $(M',g',\varphi')$ is another development and $M\preceq M'$ then $M'\cong M$.}
\end{equation}
Here the authors of \cite{CBGer69} use Zorn's lemma: it is easy to see that any chain of developments admits a maximizer (just `glue' all the developments together to create one bigger than any of those in the chain), hence Zorn's lemma ensures the existence of the desired maximizer.
\item[{\sf B)}] One proves that given two developments $M_1,M_2$ there exists a development $M$  that is $\preceq M_1$ and $\preceq M_2$ and is maximal in the  sense that:
\[
\text{if $N$ is another development $\preceq M_1$, $\preceq M_2$ and $M\preceq N$, then $N\cong M$.}
\]
Even in this case, in \cite{CBGer69} the argument makes use of Zorn's lemma by observing that any chain of `common developments' admits a maximizer that is still a common development.
\item[{\sf C)}] One proves that if $M_1$ and $M_2$ are two developments and $M$ is a maximal common development in the sense of ${\sf (B)}$ above, then gluing $M_1$ and $M_2$ along $M$ produces a development, that obviously is $\succeq M_1$ and $\succeq M_2$. Here the difficult part, is in showing that the result of the gluing is Hausdorff. Very shortly and roughly said, if it were not there would $p_1\in M_1$ and $p_2\in M_2 $ that do no possess disjoint neighbourhoods   in the gluing. It is easy to see that this can occur only if they belong to respective boundaries of the glued regions, but then - with some work that relies on local uniqueness of solutions of Einstein's equations - one can show that the common development can be extended to incorporate also $p_1,p_2$, contradicting  maximality.
\item[{\sf D)}] The conclusion is now easy: we know that there is a development  $M$ that is  maximal in the sense of ${\sf (A)}$. Let $M'$ be another development, $N$ a maximal common development as in ${\sf (B)}$ and use ${\sf (C)}$ to glue $M$ and $M'$ along $N$. This new development is $\succeq M$ and thus, by maximality of $M$, must be isomorphic to $M$. It is easy to see that for this to happen we must have that $N$ is the whole $M'$, i.e.\ that $M'\preceq M$. In other words, $M$ is the desired development maximal in the sense $(1)$ of the statement.
\end{itemize}
Points ${\sf (C),(D)}$ in this argument do not require any Choice at all, so we won't give further  details and refer rather to the original source and to the comprehensive text \cite{Ring09}. We shall instead focus on how the `small steps' and `big steps' arguments can be used to tackle points ${\sf (A),(B)}$.

\begin{proofACDC} We start with proving point ${\sf (A)}$.  Let $A$ be the collection of all developments of $(\Sigma,h,\kappa)$. We avoid the usual set-theoretic issues related to the ``all'' above by asking  the underlying set on which the manifold structures are given to be a fixed set of cardinality of continuum (notice that Proposition \ref{prop:geroch}  ensures that this is the cardinality of any connected manifold equipped with a smooth metric tensor).  Let $F\subset A$ be the collection of all the developments satisfying \eqref{eq:altmax}  above. Let also $\mathcal S$ be the collection of $\alpha$-sequences $(M_\beta)_{\beta<\alpha}$ in $A$, $\alpha<\omega_1$, that are strictly increasing, meaning that for any $\beta_1<\beta_2<\alpha$ we have $M_{\beta_1}\preceq M_{\beta_2}$ and $M_{\beta_2}\not\preceq M_{\beta_1}$. We show that $A,F,\mathcal S$ satisfy the assumptions in Lemma \ref{le:ACDC}.

Theorem \ref{thm:local}  ensures that $\mathcal S$ is not empty, as it contains a 1-sequence consisting of one development. Property  $(ii)$ holds trivially. $(i)$ is obvious by definition if $\alpha$ is a successor ordinal. If instead is a limit ordinal, then $M_\alpha$ can be built via a gluing procedure as done in \cite{CBGer69}, that requires no choice. More precisely, in \cite{CBGer69} the following has been shown: if we are given a totally ordered collection of developments, then there exists a development bigger than, or isomorphic to, all of these. Such development is built by gluing the given ones along the maps $\psi$ coming with the relation $\preceq$. The existence of such development immediately grants that property $(i)$ holds also for $\alpha$ limit, Moreover, together with Proposition \ref{prop:geroch}, it also implies the upper bound $(iii)$. Indeed, given a strictly increasing $\alpha$-sequence $(M_{\beta})_{\beta<\alpha}$ for some ordinal $\alpha$, possibly $\geq\omega_1$, what just said allows to realize, via the maps $\psi$, the $M_\beta$'s as subsets of a bigger development $M$ and the sets $U_\beta:=M_{\beta+1}\setminus\{\text{closure of $M_\beta$ in $M$}\}$ are open, disjoint and not empty (recall \eqref{eq:podev}). Proposition \ref{prop:geroch} ensures that $M$ is separable, thus we can only have a countable collection of such $U_\beta$'s, i.e.\ $\alpha$ is countable (alternatively, without using \eqref{eq:podev} we can argue as in the footnote in the introduction to conclude that the increasing $\alpha$-sequence $(M_\beta)_{\beta<\alpha}$ of open subsets of $M$ must have countable cofinality, proving that $\alpha$ is countable).

It follows by Lemma \ref{le:ACDC}  that $F$ is not empty, i.e.\ that there is a development satisfying \eqref{eq:altmax}, as desired.

We turn to point ${\sf (B)}$. We want to show that a development as in \eqref{eq:altmax}  is maximal also in the sense of $(1)$ in the statement.  This follows   along the very same arguments just used: we let $A$ be the collection of developments, $F$ that of those satisfying the maximality property just stated and $\mathcal S$ the collection of strictly increasing $\alpha$-sequences in $A$ bounded from above by both $M_1$ and $M_2$. The non-emptiness of $\mathcal S$ follows by Theorem \ref{thm:local} and then that of $F$ by Lemma \ref{le:ACDC}. 
\end{proofACDC}
\begin{proofSTEPS} Let us define a way to quantify developments. Let $D\subset T\Sigma$ be countable and dense, let 
\[
D^T:=\{(t,v):\ t\in\Q,\ v\in D\text{ and }\sqrt{h(v,v)}<t\}
\]
and let $(t_n,v_n)_{n\in\N}$ be an enumeration of $D^T$. Also, let us fix an order isomorphism $\eta:[0,+\infty]\to[0,1]$.

Let $(M,g,\varphi)$ be a development of $(\Sigma,h,\kappa)$, recall that $\varphi$ induces, for any $p\in\Sigma$, an isomorphism $\Phi$ of $(\R\times T_p\Sigma, dt^2-h)$ and $(T_{\varphi(p)}M,g)$ sending $(1,0)$ to a future vector. For each $n\in\N$, $\Phi(t_n,v_n)$ is a future timelike vector in $M$: let  $\gamma_n=\gamma_n(M)$ be the maximal geodesics (intended as solution of the geodesic equation) in $M$ with $\gamma_n'(0)=\Phi(t_n,v_n)$. Then smoothness of $M$ grants that this is a good definition. Write $(-a_n,b_n)\subset\R$ with $a_n,b_n>0$ for the interval of definition of $\gamma_n$. Then define
\begin{equation}
\label{eq:quantMGHD}
F(M)=F(M,g,\varphi):=\sum_{n\in\N}\tfrac1{2^n}(\eta(b_n)+\eta(a_n))\ \in[0,4].
\end{equation}
By definition it is obvious that $F$ is monotone, i.e.\ that if $(M_1,g_1,\varphi_1)$ and $(M_2,g_2,\varphi_2)$ are two developments with $M_1\preceq M_2$ then $F(M_1)\leq F(M_2)$ (because the domain of $\gamma_n(M_1)$ is contained in that of $\gamma_n(M_2)$ for every $n\in\N$). We now observe that $F$ is strictly monotone, i.e.
\begin{equation}
\label{eq:Fstrict}
M_1\preceq M_2\text{ and }F(M_1)=F(M_2)\qquad\Rightarrow\qquad M_1\cong M_2.
\end{equation}
To see this we shall prove that if $M_2$ is a strict extension of $M_1$ (i.e.\ the canonical map $\psi:M_1\to M_2$ coming from $M_1\preceq M_2$ is not surjective), then $F(M_2)>F(M_1)$. By what already observed and the fact that $\eta$ is an isomorphism, to see this it suffices to find $n\in\N$ such that the interval of definition of $\gamma_n(M_2)$ strictly contains that of $\gamma_n(M_1)$.

To see this,  recall by \eqref{eq:podev} that $M_2\setminus\overline{\psi(M_1)}$ is not empty. Pick $p\in M_2\setminus\overline{\psi(M_1)}$ and then an inextendible timelike geodesic $\gamma $ passing through $p$. Since $\varphi_2(\Sigma)$ is a Cauchy hypersurface of $M_2$, such $\gamma$ must intersect $\varphi_2(\Sigma)$. Approximating the speed of $\gamma$ at the intersection point with elements of $D^T$, by continuity we find $n\in\N$ such that the image of $\gamma_n(M_2)$ intersects $M_2\setminus\overline{\psi(M_1)}$. Since $\psi\circ\gamma_n(M_1)$ coincides with the restriction of $\gamma_n(M_2)$ to the domain of definition of $\gamma_n(M_1)$, this proves that such domain of definition is strictly contained in that of $\gamma_n(M_2)$, as desired. 

Having established the strict monotonicity \eqref{eq:Fstrict} we proceed by proving points ${\sf (A),(B)}$.  We start with  ${\sf (A)}$. 
Given a development $(M,g,\varphi)$ we define $G(M)\in[0,4]$ as
\[
G(M):=\sup\big\{F(M')\ :\ (M',g',\varphi')\text{ is a development with }M\preceq M'\big\}
\]
and notice that $M_1\preceq M_2$ implies $G(M_1)\geq G(M_2)$. Now define recursively a sequence $(M_n)$ of developments as follows. Let $(M_0,g_0,\varphi_0)$ be an arbitrary development as given  by Theorem \ref{thm:local}. Having defined $(M_n,g_n,\varphi_n)$ pick $(M_{n+1},g_{n+1},\varphi_{n+1})$ so that $M_n\preceq M_{n+1}$ and $F(M_{n+1})\geq \tfrac12(F(M_n)+G(M_n))$. The definition of $G(M)$ ensures that such $M_{n+1}$ exists. Gluing all these developments (as in \cite{CBGer69}), we find a development $(M,g,\varphi)$ with $M_n\preceq M$ for every $n\in\N$. The construction ensures that $G(M_{n+1})-F(M_{n+1})\leq G(M_n)-\tfrac12(F(M_n)+G(M_n))\leq \tfrac12(G(M_n)-F(M_n))$ and thus $G(M_n)-F(M_n)\leq\tfrac1{2^n}(G(M_0)-F(M_0))\to 0$, so that the monotonicities of $F,G$ imply that $G(M)=F(M)$, which by \eqref{eq:Fstrict} means that $M$ is maximal in the sense of \eqref{eq:altmax}.

For  ${\sf (B)}$ we can argue along the lines just used by  quantifying the size of common developments using $F$. We omit the details.
\end{proofSTEPS}

\begin{remark}[The longest line]\label{re:longest}{\rm
The long line $L$ is the topological space  whose underlying set is $\omega_1\times[0,1)$ equipped with the order topology, where the (total) order is defined by
\[
(\alpha,t)\leq(\beta,s)\qquad\text{whenever}\qquad\text{either $\alpha<\beta$ or ($\alpha=\beta$ and $t\leq s$)}.
\]
It is quite easy to see that  $L\setminus\{(0,0)\}$ is locally homeomorphic to $\R$. A quick way to realize this is by recalling that any countable ordinal can be embedded in $\R$: from this fact it is immediate to see that for any point $(\bar\alpha,\bar t)\in L$ the set $\{(\alpha,t)\in L:(\alpha,t)\neq (0,0), \ (\alpha,t)\lneq (\bar\alpha,\bar t)\}$ is order isomorphic to some open interval in $\R$, and thus homeomorphic to such interval (as the topology on $L$ is the order topology). It is also clear that $L$ is not separable: if $((\alpha_n,t_n))_{n\in\N}$ is any given sequence in $L$, then for $\bar\alpha:=\sup_n\alpha_n+1$ the open ray of points $\gneq (\bar\alpha,0)$ does not contain any of these.

Notice that $L$ is the longest a line could be, i.e.\ that the following holds:
\begin{equation}
\label{eq:longest}
\begin{split}
&\text{Let $\X$ be a Hausdorff connected topological space locally homeomorphic to $\R$.}\\
&\text{Then there is a non-decreasing sequence $(U_\alpha)_{\alpha<\omega_1}$ of separable open subsets}\\
&\text{so that $\X=\cup_{\alpha<\omega_1}U_\alpha$.}
\end{split}
\end{equation}
Notice that this implies that an $\X$ as in the statement is homeomorphic to either $\R$ or to $L\setminus\{0\}$ or to two copies of $L$ glued in $0$.

To prove the above one line of thought is to notice that 
\begin{equation}
\label{eq:notice}
\text{$U$ connected, separable open subset of $\X$}\quad\Rightarrow\quad\text{$\bar U$ has the Lindel\"of property.}
\end{equation}
The claim is trivial, as $U$ must be homeomorphic to an open interval, thus $\bar U$ is homeomorphic to either an open, or half-open or closed interval, and the conclusion is true in either case.

Given \eqref{eq:notice}, the above claim \eqref{eq:longest} can be proved along the following lines. Let $U_0$ be an arbitrary non-empty, connected and separable open subset of $\X$. As mentioned, this is homeomorphic to an interval and thus has the Lindel\"of property.

We are going to recursively define $(U_\alpha)_{\alpha<\omega_1}$ as in \eqref{eq:longest} by also ensuring that each $U_\alpha$ has the Lindel\"of property. If $\alpha$ is a limit ordinal we put $U_\alpha:=\cup_{\beta<\alpha}U_\beta$ and notice that being a countable union of separable sets with the Lindel\"of property, $U_\alpha$ is separable and with the Lindel\"of property. On the other hand, given $U_\alpha$ we define $U_{\alpha+1}$ as follows. Each $p\in\overline  U_\alpha$ has a neighbourhood $V_p$ separable and with the Lindel\"of property (being homeomorphic to an interval). Also, by the Lindel\"of property ensured by \eqref{eq:notice} there are $(p_n)_{n\in\N}\subset \overline  U_\alpha$ such that $\overline  U_\alpha\subset\cup_nV_{p_n}$. We define $U_{\alpha+1}:=\cup_nV_{p_n}$ and notice that it is connected, separable and with the Lindel\"of property.

To conclude we need to prove that $\cup_{\alpha<\omega_1}U_\alpha=\X$ and since $\X$ is connected and such union is open , it suffices to prove that it is also closed. Let thus $p\in \overline{\cup_{\alpha<\omega_1}U_\alpha}$ and notice that since $\X$ is locally homeomorphic to $\R$ there is a sequence $(p_n)\subset \cup_{\alpha<\omega_1}U_\alpha$ converging to $p$. Let $\alpha_n<\omega_1$ be so that $p_n\in U_{\alpha_n}$ and let $\alpha:=\sup_n\alpha_n<\omega_1$. We then have $p\in \overline U_\alpha\subset U_{\alpha+1}$, as desired.

It is natural to wonder whether the same result as in \eqref{eq:longest} holds in higher dimensions. Contrary to what I initially thought, the answer is no. A relevant example here is the Pr\"ufer surface (see e.g.\ \cite[Appendix A]{SpivakI}): it  is  a 2-dimensional smooth (in fact analytic) connected manifold containing a family of mutually disjoint open sets $(U_a)$ parametrized by a real parameter $a\in\R$. If a result like \eqref{eq:longest} holds for such surface, then we could build a surjective map from $\omega_1$ to $\R$, proving the Continuum Hypothesis. Since this is independent from ZFC, we see that the higher dimensional analogue of \eqref{eq:longest} is a much more delicate result. 
}\fr\end{remark}

\appendix
\section{The first uncountable ordinal}

For convenience of the reader unfamiliar with the matter, we recall here few facts about ordinal numbers. The aim of these few pages is to convince those who feels a shiver down the spine in reading `transfinite recursion' or  `first uncountable ordinal' that ordinal numbers are innocent creatures and that  the first uncountable one $\omega_1$ is very concretely constructible (in fact, this is just the collection of all countable ordinals, much like the collection of all finite numbers is the `first infinite number'). In particular, we stress that existence of ordinal numbers has nothing to do with the Axiom of Choice (that instead is about existence of bijections from  arbitrary sets to ordinals) and that the constructibility of $\omega_1$ has nothing to do with the Continuum Hypothesis (that instead is about  the existence of a bijection between $\R$ and $\omega_1$).

Our presentation closely follows that in \cite{Jech2003}, to which we also refer for (very much) more on the matter. We emphasise  that in there  ordinal numbers are the first concept introduced  after the axioms of set theory. 

In order to avoid dealing with the (non-existing) \emph{set} of all ordinals, we shall just focus on countable ordinals and the first uncountable one: this will spare us from dealing with the set-theoretic issues related to `collections of objects too big to be sets' and is  sufficient for all the uses of ordinals made in this text\footnote{Not that there is any  issue in doing so: the collection of all ordinals is not a set, but rather a proper class. `Sets' and `proper classes' share the same properties with one key difference: proper classes are never elements of other objects. This solves Russell's paradox because the collection of all sets, or of all ordinals, is not a set but rather a proper class and as such it does not belong to itself (nor to anything else).

In other words, the barber in Russell's village is a woman.}.

\bigskip

Ordinal numbers are particular well-ordered sets, so to understand the former better to first discuss the latter. A \emph{well order} relation $\leq $ on a set $W$ is  a total order such that every non-empty subset has a minimum. As usual, we shall write $y<x$ to intend $y\leq x$ and $y\neq x$. Given a well order $(W,\leq)$ and $x\in W$, the \emph{initial segment} $W(x)\subset W$ is the set $\{y\in W:y< x\}$.
 A \emph{lower set} of a total order $W$ is a subset $L\subset W$ such that $x\in L$ and $y\leq x$ implies $y\in L$.
\begin{itemize}
\item[i)] If $W$ is a well order the map $x\mapsto W(x)$ is an order isomorphism from $W$ to the collection of proper lower subsets, ordered by inclusion.\\
\fs{\emph{Proof}\ The fact that the given map is an injective monotone map is obvious. For surjectivity, let $L\subset W$ be a proper lower set and $x$ the minimal element in $W\setminus L$. Then $L=W(x)$, obviously.\penalty-20\null\hfill$\square$
}
\item[ii)] If $W$ is a well order and $f:W\to W$ strictly  increasing, then $x\leq f(x)$ for every $x\in W$.
\fs{\emph{Proof}\ Assume $\{x\in W:f(x)<x\}$ to be not empty and let $x_0$ be its minimal element. Then $f(x_0)<x_0$ and thus $f(f(x_0))<f(x_0)$, contradicting the minimality of $x_0$.\penalty-20\null\hfill$\square$
}
\item[iii)] The only authomorphism of a well ordered set is the identity.\\
\fs{\emph{Proof}\ Apply {\rm (ii)} above to the authomorphism and its inverse.\penalty-20\null\hfill$\square$}
\item[iv)] If two well order are isomorphic, the isomorphism is unique.\\
\fs{\emph{Proof}\ Direct consequence of ${\rm (iii)}$ above.\penalty-20\null\hfill$\square$}
\item[v)] No well order is isomorphic to an initial segment of itself.\\
\fs{\emph{Proof}\ If the image of the isomorphism $f$ is $W(x)$, then $f(x)\in W(x)$ and thus $f(x)<x$.\penalty-20\null\hfill$\square$}
\item[vi)] If $W_1,W_2$ are two well orders, then one and only one of the following occurs:
\begin{itemize}
\item[a)] $W_1$ and $W_2$ are isomorphic,
\item[b)] $W_1$ is isomorphic to an initial segment of $W_2$,
\item[c)] $W_2$ is isomorphic to an initial segment of $W_1$.
\end{itemize}
\fs{\emph{Proof}  Define $\Gamma\subset W_1\times W_2$ as
\[
\Gamma:=\{(x,y)\in W_1\times W_2: W_1(x)\text{ is isomorphic to }W_2(y)\}
\]
and let $U_1,U_2$ be the images of the projections of $\Gamma$ in $W_1,W_2$ respectively. Using the points above it is clear that $\Gamma$ is (the graph of) a bijective order isomorphism $f$ from $U_1$ to $U_2$. Suppose that $U_1,U_2$ are both proper subsets of $W_1,W_2$ respectively and let $x\in W_1$ be the least element of $W_1\setminus U_1$ and $y\in W_2$ the least element in $W_2\setminus U_2$. Then the restriction of $f$ to $W_1(x)$ is an order isomorphism which, by the points above, has image precisely $W_2(y)$. In other words we have $(x,y)\in\Gamma$, contradicting the definitions of $x,y$. Thus either $U_1=W_1$ or $U_2=W_2$ or both are true, giving the claim.\penalty-20\null\hfill$\square$} 
\end{itemize}
We now define countable ordinals (i.e.\ ordinals whose underlying set is at most countable) as   representatives of isomorphism classes of well orders. This is in principle in contrast with our initial claim that ordinals are themselves well ordered sets, but this discrepancy (that only occurs because of the way  we chose to construct ordinals - see Remark \ref{re:sdeford}) will soon disappear, see point (vii) below.

Fix an infinite countable set $C$, e.g.\ $C=\N$, consider the set $\mathcal P(C\times C)$ of subsets of $C\times C$ and notice that, by definition, these are all the relations on $C$. Let $\WO(C )\subset \mathcal P(C\times C)$ be the collection of those relations that are well orders on either the whole $C$ or some subset of it (e.g.\ we see the singleton $\{(x,x)\}$ for $x\in C$ as the well order relation on the singleton $\{x\}$, rather than as a relation on the whole $C$).  We quotient $\WO(C)$  by the natural isomorphism relation $\sim$ of well orders and \emph{define} countable ordinals as $\sim$-equivalence classes, i.e.\ as elements of $\WO(C )/\sim$. For $\alpha,\beta\in \WO(C)$ we say that $\alpha\leq \beta$ whenever either $\alpha=\beta$ or some (and thus every) element in the equivalence class $\alpha$ is isomorphic to an initial segment of some  (and thus every) element in the equivalence class  $\beta$. We also put $\omega_1:=\WO(C )/\sim$ and observe  that by item (vi) above $\leq$ is a total order on  $\omega_1$. For $\alpha\in\omega_1$ we put $\hat\alpha:=\{\beta\in\omega_1:\beta<\alpha\}$.

Notice that for any natural number $n$, up to isomorphism there is exactly one total order of the set of $n$ elements and such order is a well order: with a standard abuse of notation we shall denote by $n$ the corresponding equivalence class, i.e.\ the associated ordinal. Notice that if $n<m$ as natural numbers, then rather trivially the well order of $n$ elements is isomorphic to an initial segment of the well order with $m$ elements, so that the standard ordering on natural numbers coincides with that of the corresponding ordinal.

We have:
\begin{itemize}
\item[vii)] For $\alpha\in\omega_1$ the set $\hat\alpha$ equipped with the order it inherits from $\omega_1$   is isomorphic to any well order in the equivalence class $\alpha$. In particular $\hat\alpha$ is well ordered.\\
\fs{\emph{Proof}\ Direct consequence of the definitions and of ${\rm (i)}$.
\penalty-20\null\hfill$\square$
}
\item[viii)] $(\omega_1,\leq)$ is a well order.\\
\fs{\emph{Proof}\ Let $A\subset\omega_1$ be not empty and $\alpha\in A$. Then either  $\alpha$ is the minimum of $A$ or $\{\beta\in A:\beta< \alpha\}\neq\emptyset$  and, being it contained in the well ordered set $\hat\alpha$,  has minimum.
\penalty-20\null\hfill$\square$
}
\item[ix)] $\omega_1$ is not countable.\\
\fs{\emph{Proof}\ As $\WO(C)$ is the collection of all countable well orders, if $\omega_1$  were countable it would be isomorphic to some element of  $\WO(C)$ and thus, by (vii), to an initial segment of itself.
\penalty-20\null\hfill$\square$
}
\item[x)] If $W$ is an uncountable well order, then $\omega_1$ is either isomorphic to $W$ or to an initial segment of it.\\
\fs{\emph{Proof}\ By ${\rm (vi)}$ it suffices to exclude that $W$ is isomorphic to an initial segment of $\omega_1$. But this is obvious, as by construction and ${\rm (vii)}$ all initial segments of $\omega_1$ are countable.
\penalty-20\null\hfill$\square$
}
\item[xi)] A subset of $\omega_1$ has supremum in $\omega_1$ if and only if it is   countable.\\
\fs{\emph{Proof}\ Let $(\alpha_n)\subset \omega_1$ be countable. By (vii) for every $n$ the set $\hat\alpha_n$  is countable, hence so is $\cup_n\alpha_n$\footnote{proving that countable union of countable sets is countable relies on the Axiom of Countable Choice,  that is weaker than Countable Dependent Choice that I am assuming throughout the text.}. It follows that there is  $\alpha\in\omega_1$ such that $\beta< \alpha$ for every $\beta\in\cup_n\hat\alpha_n$. Therefore  $\alpha\geq\alpha_n$ for every $n$ (for if not $\alpha< \alpha_n$ for some $n$, implying $\alpha\in\cup_n\hat\alpha_n$ and thus $\alpha<\alpha$), i.e.\ $\alpha$ is an upper bound for the original sequence. Thus the set of upper bounds is not empty: its minimum is the desired supremum in $\omega_1$. 
Conversely,  any bounded set  $A\subset\omega_1$ is, by definition, contained in  a set of the form $\{\beta\in\omega_1:\beta\leq \alpha\}$ for some $\alpha\in\omega_1$, which is clearly countable.
\penalty-20\null\hfill$\square$
}
\end{itemize}
We have not defined general ordinals, but ${\rm (viii),(ix),(x)}$  above should make it clear why it is fair, even for us, to say that $\omega_1$ is the least uncountable ordinal (in our case this is  interpreted as $\omega_1$ being the least uncountable well order). We can now state a result often used in this manuscript. In the literature it is usually formulated under the additional assumption that the map is continuous or as the non-existence of an injective monotone map from $\omega_1$ to $\R$.
\begin{proposition}\label{prop:omega1r}
Any monotone map from $\omega_1$ to $\R$ is eventually constant (and bounded). 
\end{proposition}
\begin{proof}
Let $f:\omega_1\to\R$ be monotone and $S:=\sup_{\alpha\in\omega_1}f(\alpha)\in\R\cup\{+\infty\}$. Then  there is a sequence $(\alpha_n)\subset\omega_1$ such that $S=\sup_nf(\alpha_n)$. Hence  $\alpha:=\sup_n\alpha_n$ is a countable ordinal (recall ${\rm (xi)}$ above) and by monotonicity $f(\alpha)\geq f(\alpha_n)$ for every $n$, i.e.\ $f(\alpha)\geq S$. Then for every $\alpha\leq\beta<\omega_1$ we have $S\leq f(\alpha)\leq f(\beta)\leq S$, the last inequality being a consequence of the definition of $S$, proving that $\R\ni f(\beta)=S$.
\end{proof}
Although not necessary for our discussion, we point out that  the order relation of $\R$ (in fact also of $\Q$) is rich enough to allow embeddings of any countable ordinal:
\begin{itemize}
\item[xii)] Let $\alpha\in\omega_1$. Then there is  an order isomorphism  $f:\hat\alpha\to\R$  of $\hat\alpha$ with its image in $\R$.\\
\fs{\emph{Proof}\  Let $(\beta_n)$ be an enumeration of $\hat\alpha$.  Define $f:\hat\alpha\to \Q$ inductively as follows. Set $f(\beta_0):=0$ and, assuming to have defined $f(\beta_0),\ldots, f(\beta_n)$, define $f(\beta_{n+1})$ so that it has the same order relations with the previously defined values as $\beta_{n+1}$ has with $\beta_0,\ldots,\beta_n$. The fact that this is always possible is trivial.

Alternatively, we can proceed by transfinite induction (see  below) noticing that:
\begin{itemize}
\item[i)] $0$, i.e.\ the empty set, can be embedded into $\R$
\item[ii)] If $\alpha$ can be embedded, so can $\alpha+1$. To see this, fix an order isomorphism of $\R$ and $(0,1)$, use it to embed $\alpha$ in $(0,1)$ and then send the extra element to 1.
\item[iii)] If $n\mapsto\alpha_n$ is an increasing sequence of ordinals embeddable on $\R$, then so is $\alpha:=\sup_n\alpha_n$. To see this, notice that if $\alpha_1<\alpha_2$ are both embeddable and we have fixed an embedding of $\alpha_1$ in $(0,1)$, then we can - easily - embed $\alpha_2$ in $(0,2)$ in such a way that this embedding extends the one previously given for $\alpha_1$. Iterating this procedure we see that we can find embeddings $f_n:\hat\alpha_n\to(0,n)$ such that $f_n\restr{\hat\alpha_{n-1}}=f_{n-1}$. Then for every $\beta<\alpha$ and every $n,m\in\N$ big enough so that $\beta<\alpha_n$ and $\beta<\alpha_m$ we have $f_n(\beta)=f_m(\beta)$. Therefore the map $f:\hat\alpha \to \R$ defined as $f(\beta):=f_n(\beta)$ where $n=n(\beta)$ is big enough so that $\beta<\alpha_n$ is well defined and the desired embedding.
\end{itemize}
Since every limit countable ordinal is, being countable, the supremum of an increasing sequence, this suffices to conclude.
\penalty-20\null\hfill$\square$
}
\end{itemize}

We  now discuss \emph{transfinite induction} and \emph{transfinite recursion}. We start with induction:
\begin{itemize}
\item[xiii)] Let $A\subset\omega_1$ be so that  for every $\alpha\in\omega_1$ we have ``$\hat\alpha\subset A$ implies $\alpha\in A$''. Then $A=\omega_1$.\\
\fs{Proof. If $\omega_1\setminus A$ is not empty, there must be a least element $\alpha$. For such $\alpha$ we have $\hat\alpha \subset A$ and $\alpha\notin A$, contradicting the assumption.\penalty-20\null\hfill$\square$
}
\end{itemize}
Notice that formulating the (transfinite) induction assumption as above frees us from checking the 0 case: we surely have $\emptyset\subset A$ and thus by assumption we must have $0\in A$. Thus $\{0\}\subset A$, implying $1\in A$ `and so on' in the sense made rigorous by the proof above. 

Still, in applications the induction assumption is often verified by distinguishing which kind of ordinal $\alpha$ is; this also clarifies the relation between transfinite and classical induction. We say that a (countable) ordinal is a \emph{successor ordinal} if there is $\beta<\alpha$  so that $\alpha$ is the least ordinal bigger than $\beta$. In this case we write $\alpha=\beta+1$ (a terminology also justified by how one defines addition on ordinals, something we will not discuss). Ordinals that are not successor ordinals are called \emph{limit ordinals}. Notice that:
\begin{itemize}
\item[xiv)] $\alpha$ is a limit ordinal if and only if it is the supremum of the set $\hat\alpha\subset\omega_1 $ of ordinals $\beta<\alpha$.\\
\fs{\emph{Proof}\ Clearly if $\alpha=\beta+1$ then $\beta$ is the supremum of $\hat \alpha$. Conversely, if $\beta:=\sup\hat\alpha<\alpha$, then there is no ordinal $\gamma$ with $\beta<\gamma<\alpha$, hence $\alpha=\beta+1$.
\penalty-20\null\hfill$\square$}
\end{itemize}
In this terminology 0 is a limit ordinal, but due to its peculiarity it is often treated differently from non-zero limit ones. Then arguments by transfinite induction are often phrased as:
\begin{itemize}
\item[xv)] Let $A\subset \omega_1$ be such that $0\in A$, ``if $\alpha\in A$ then $\alpha+1\in A$'' and ``if $\alpha$ is a limit ordinal such that $\beta\in A$ for every $\beta<\alpha$ then $\alpha\in A$''. Then $A=\omega_1$.\\
\fs{\emph{Proof}.\ Obvious from the previous formulation and the discussion just had.\penalty-20\null\hfill$\square$
}
\end{itemize}
We now turn to \emph{transfinite recursion}, something that is typically used to construct objects.  Recall that given a set $A$, by $A^{<\omega_1}$ it is meant the set of functions from countable ordinals to $A$. 
\begin{itemize}
\item[xvi)] Let $A$ be a  set  and ${\sf Next}: A^{<\omega_1}\to A$ a given map. Then there exists a unique map $f:\omega_1\to A$ such that:
\begin{equation}
\label{eq:trrec}
f(\alpha)={\sf Next}(f\restr{\hat\alpha})\qquad\forall\alpha\in\omega_1.
\end{equation}
\fs{\emph{Proof}.\  Let $B\subset \omega_1$ be the set of all $\beta\in\omega_1$ such that: there exists a unique map $f:\{\alpha\leq\beta\}\to A$ such that \eqref{eq:trrec} holds for all $\alpha\leq \beta$. Let $\gamma\in\omega_1$ be such that  $\hat\gamma\subset B$: in this case $f$ is already uniquely defined on $\hat\gamma$ and the constraint/definition \eqref{eq:trrec} forces the value of $f(\gamma)$, proving that $\gamma\in B$. We thus proved that $\hat\gamma\subset  B$ implies $\gamma\in B$, hence by transfinite induction we see that $B=\omega_1$, as desired.
\penalty-20\null\hfill$\square$}
\end{itemize}
The idea behind the statement of the transfinite recursion theorem is that the map ${\sf Next}$ tells us what is the `next value' the function $f$ should take, provided the previous ones have already been assigned. In practical situations, such map might be defined distinguishing the successor/limit ordinal cases.  Notice the difference with the Axiom ${\sf DC}_{\omega_1}$: here no choice is required, because ${\sf Next}$ is a function.

\begin{remark}[The actual definition of ordinals]\label{re:sdeford}{\rm There is something clumsy about our presentation of (countable) ordinals, as we started with an arbitrary set, considered well orders in it to define ordinals and later in (vii) and (viii) realized that the set of (countable) ordinals is well ordered and that the initial segment corresponding to any given $\alpha$ in this set   is isomorphic to the well orders in the equivalence class  $\alpha$. 

Given this, one could question what is the role of the initial set and whether it is possible to use the order relation between ordinals to define ordinals themselves by declaring that
\begin{equation}
\label{eq:orddef}
\text{Each ordinal is the set of ordinals smaller than it.}
\end{equation}
This, of course, is the celebrated informal definition of ordinals as given by Zermelo and von Neumann. Intuitively, such definition works as follows. The empty set surely is a set or ordinals (none of its elements is not an ordinal, after all), thus according to the above the empty set is an ordinal. We denote it by 0 and notice that, again by the above definition, 0 is bigger than no ordinal. Now that we have 0, we can certainly consider the set $\{0\}$, which is also an ordinal: this is bigger than the sole ordinal 0 and shall be denoted 1. Then $2:=\{0,1\}$, $3:=\{0,1,2\}$,\ldots, $\omega_0:=\{0,1,2,\ldots\}$. Notice that as a set $\omega_0$  is precisely $\N$ and has cardinality $\aleph_0$, but in the theory of ordinals it is better to use a different letter, as there are other countable ordinals such as $\omega_0+1:=\{0,1,2,\ldots,\omega_0\}$. In fact, as discussed above, there are uncountably many countable ordinals.

The actual definition of ordinal, as given e.g.\ in \cite[Def. 2.10]{Jech2003}, makes rigorous   \eqref{eq:orddef} above as follows:
\begin{equation}
\label{eq:orddef2}
\text{An ordinal is a set that is transitive and well ordered by $\in$.}
\end{equation}
Here `transitive' for a set $A$ means that each element of $A$ is also a subset of $A$ and `well ordered by $\in$' means that the relation $\leq$ in $A$ defined as $B\leq C$
 if either $B=C$ or $B\in C$ is a well order. We refer to \cite{Jech2003} for more on the matter.
}\fr\end{remark}

%

\def\cprime{$'$} \def\cprime{$'$}

\end{document}